\documentclass[12pt]{amsart}
\font\emailfont=cmtt10

\usepackage{amsmath,amsthm,amsfonts,amscd,flafter,epsf}
\usepackage{graphics,todonotes,appendix}
\usepackage{epsfig}
\usepackage{psfrag}

\newcommand\commentable[1]{#1}

\newtheorem{thm}{Theorem}
\newtheorem{conjecture}{Conjecture}

\newtheorem{theorem}{Theorem}[section]
\newtheorem{prop}[theorem]{Proposition}

\newtheorem{defn}[theorem]{Definition}

\newtheorem{remark}[theorem]{Remark}

\def\endproofof{\relax\ifmmode\expandafter\endproofmath\else
  \unskip\nobreak\hfil\penalty50\hskip.75em\hbox{}\nobreak\hfil\bull
  {\parfillskip=0pt \finalhyphendemerits=0 \bigbreak}\fi}
\def\endproofofmath$${\eqno\bull$$\bigbreak}

\def\endproof{\relax\ifmmode\expandafter\endproofmath\else
  \unskip\nobreak\hfil\penalty50\hskip.75em\hbox{}\nobreak\hfil\bull
  {\parfillskip=0pt \finalhyphendemerits=0 \bigbreak}\fi}
\def\endproofmath$${\eqno\bull$$\bigbreak}
\def\bull{\vbox{\hrule\hbox{\vrule\kern3pt\vbox{\kern6pt}\kern3pt\vrule}\hrule}}

\newcommand{\Q}{\mathbb{Q}}

\newcommand{\Z}{\mathbb{Z}}

\newcommand\SpinC{\mathrm{Spin}^c}
\newcommand{\F}{\mathbb F}

\newcommand\relspinc{\underline{\spinc}}

\newcommand\Filt{\mathcal F}

\newcommand\ModSphere{\ModFlow\left({\mathbb S}\longrightarrow 
\Sym^{g-1}(\Sigma_{1})\times \Sym^2(\Sigma_{2})\right)}
\newcommand\ModSpheres\ModSphere

\newcommand\HFa{\widehat{HF}}

\newcommand\UnparModSp{\widehat \ModSp}
\newcommand\UnparModFlow\UnparModSp
\newcommand\Mod\ModSp

\newcommand{\spinc}{\mathfrak s}

\newcommand{\spinct}{\mathfrak t}

\newcommand\ModMaps{\mathcal M}
\newcommand\ModSp\ModMaps

\newcommand\spincrel\relspinc

\newcommand\Dual{\mathcal D}
\newcommand\Duality\Dual

\newcommand\ons{Ozsv{\'a}th and Szab{\'o}}
\newcommand\os{{Ozsv{\'a}th-Szab{\'o}}}

\newcommand\SSurf{S}

\newcommand\Sym{\mathrm{Sym}}

\newcommand{\cob}{F\setminus {S}}
\newcommand{\cobprime}{\kappa({F\setminus {S}})}

\commentable{

\title[{4-dimensional aspects of tight contact  3-manifolds}] 
{4-dimensional aspects of tight contact  3-manifolds}

}

\author[Matthew Hedden]{Matthew Hedden}
\address{Department of
Mathematics, Michigan State University, MI \newline
\indent{\emailfont{mhedden@math.msu.edu}}}

\author{Katherine Raoux}
\address{Department of
Mathematics, Michigan State University, MI \newline
\indent{\emailfont{raouxkat@msu.edu}}}

\thanks{MH gratefully acknowledges support from NSF grant DMS-1709016.  KR was partially supported by an AWM Mentoring Travel Grant}

\begin{document}

\begin{abstract}  In this article we conjecture a 4-dimensional characterization of tightness: a contact structure on a 3-manifold $Y$ is tight if and only if a slice-Bennequin inequality holds for smoothly embedded surfaces  in $Y\times [0,1]$.  An affirmative answer to our conjecture would imply an analogue of the Milnor conjecture for torus knots: if a fibered link $L$ induces a tight contact structure on $Y$ then its fiber surface  maximize Euler characteristic amongst all surfaces in $Y\times [0,1]$ with boundary $L$.  We provide evidence for both conjectures by proving them for contact structures with non-vanishing \os \ contact invariant.     \end{abstract}

\maketitle
\section{Introduction}
Contact structures on 3-manifolds are coarsely classified according to the famous {\em tight} versus {\em overtwisted} dichotomy.  The latter class is characterized by the presence of an overtwisted disk; that is, an embedded disk whose boundary is tangent to the contact structure in such a way that the framing provided by the contact structure agrees with that given by the disk.   This contact framing, viewed as an integer by way of the linking number, is the well-known Thurston-Bennequin number. Thus a contact structure being overtwisted is equivalent to possessing a Legendrian unknot with Thurston-Bennequin number zero. Using Eliashberg's classification theorem for overtwisted contact structures \cite{Eliashberg}, one can  easily see that any link type in an overtwisted contact structure admits Legendrian representatives with Thurston-Bennequin invariant any integer, see e.g. \cite[Proposition 22]{tbbounds}.  By stabilizing, however, one can also show that in {\em any} contact structure, the Thurston-Bennequin invariant can be arranged to be arbitrarily negative.  From these considerations, it follows that tightness is characterized by the existence of an upper bound on the Thurston-Bennequin numbers of Legendrian representatives of any single link type (say, the unknot).  Eliashberg and Bennequin showed that, in fact, establishing such an upper bound for a single link type implies a bound for all types in terms of the Euler characteristic of surfaces bounded by the link:
\vskip0.1in
\noindent {\bf Eliashberg-Bennequin inequality} \cite{EliashbergTB,Bennequin}.{\em \; $(Y,\xi)$ is tight if and only if
   \[ tb_\xi(\partial \Sigma)+ |rot_{[\Sigma]}(\partial \Sigma)| \le -\chi({\Sigma})\]
 \noindent  for any smoothly embedded surface  $\Sigma \hookrightarrow Y$ with Legendrian boundary.}
  \vskip0.1in
\noindent In the above, and throughout, we implicitly assume that contact structures are oriented, and that  3-manifolds are compatibly oriented.  With that in mind,  $rot_{[\Sigma]}(\partial \Sigma)$ is the other ``classical" invariant of a Legendrian knot, the {\em rotation number}, defined as the relative Euler number of the contact structure restricted to the surface $\Sigma$.  Using Kronheimer and Mrowka's affirmative answer to Milnor's question about the unknotting number of torus knots \cite{KM1}, Rudolph showed that for Legendrian links in the standard contact structure on the $3$-sphere, the Eliashberg-Bennequin bound can be strengthened to the so-called ``slice-Bennequin"  inequality, for which the bound is given by $-\chi(\Sigma)$ with $\Sigma$ a smooth and properly embedded surface in the $4$-ball.  Note that while Rudolph used Kronheimer-Mrowka's theorem to prove the slice-Bennequin inequality, the slice-Bennequin inequality itself implies the Milnor conjecture.  This indicates that the Legendrian knot theory of the tight contact $3$-sphere is closely related to smooth 4-dimensional topology.   The purpose of this article is to conjecture that tightness is, in general, a 4-dimensional phenomenon:
\begin{conjecture}[\bf General slice-Bennequin inequality]\label{conj1} $(Y,\xi)$ is tight if and only if 
\[  tb_\xi(\partial \Sigma)+ |rot_{[\Sigma]}(\partial \Sigma)| \le -\chi({\Sigma})\]
for any smoothly embedded surface  $\Sigma \hookrightarrow Y\times [0,1]$ with Legendrian boundary in $Y\times\{1\}$.
\end{conjecture}

One could imagine importing some of the tools used in the proof of the $3$-dimensional Eliashberg-Bennequin bound to attack the problem at hand; for instance, one could attempt to manipulate characteristic foliations imprinted on  singular surfaces in $Y$ obtained by projection along the interval. Such attempts, if successful, would give rise to a markedly different proof of the slice-Bennequin bound for links in $S^3$. This would, in turn, yield a far more geometric proof of the Milnor conjecture than any known to date (which rely on gauge theory, Floer homology, or TQFT structures present in link homology theories).

As evidence for the conjecture, we establish it for the  class of contact manifolds whose \os \ contact invariant (see \cite{Contact} for its definition) is non-trivial.  

\begin{thm}\label{sB} Suppose $(Y,\xi)$ is a closed contact $3$-manifold with non-vanishing \os \ contact invariant. Then
\[  tb_\xi(\partial \Sigma)+ |rot_{[\Sigma]}(\partial \Sigma)| \le -\chi({\Sigma})\]
for any smoothly embedded surface  $\Sigma \hookrightarrow Y\times [0,1]$ with Legendrian boundary in $Y\times\{1\}$.
\end{thm}

It is known that the contact manifolds with non-vanishing invariant form a proper subset of  tight contact manifolds, and  the extent to which the two classes differ is an interesting subject  \cite{BaldwinVelaVick, Ghiggini, GHVHM, Kutluhan, Massot}.  We remark that the inequality above has a natural extension to rationally null-homologous knots, which we prove in Theorem \ref{thm:rational} below.   One should also compare our result to a Bennequin bound for surfaces in compact $4$-manifolds with $\partial W=Y$ and non-trivial Seiberg-Witten invariant relative to $\xi$, proved by Mrowka and Rollin \cite{MrowkaRollin}.  The techniques at hand  easily establish an analogue of this latter bound for the \os\ contact invariant, see Theorem \ref{MRbound}.

As mentioned, the slice-Bennequin inequality for knots in $S^3$ immediately implies the Milnor Conjecture on the unknotting number of torus knots.  Indeed, we have the string of inequalities:
\[  tb(T_{p,q})+ rot(T_{p,q}) \le 2g_4(T_{p,q})-1  \le 2u(T_{p,q})-1.\]
On the other hand, it is easy to produce a Legendrian representative for $T_{p,q}$ for which the 3-dimensional Bennequin bound is sharp:
\[ 2g(T_{p,q})-1 = pq-p-q= tb(T_{p,q})+ rot(T_{p,q}). \]
It follows that the 3- and 4-genera of these knots agree, and equal the unknotting number (exhibited by an explicit sequence of $\frac{(p-1)(q-1)}{2}$ crossing changes).

More generally, one can import the above reasoning to show that the fiber surface for any fibered link $L$ whose associated open book decomposition induces the standard contact structure on the 3-sphere minimizes $-\chi(\Sigma)$ amongst all smooth and properly embedded surfaces in the 4-ball with boundary $L$.\footnote{In fact, such knots also share the property with torus knots that their fibers are properly isotopic to a piece of complex algebraic curve in the 4-ball \cite{SQPfiber}.}    Our conjecture has, as a corollary, the following analogue for arbitrary tight structures:

\begin{conjecture}[\bf General Milnor conjecture] \label{conj2} Suppose $L\subset Y$ is a fibered link with fiber $\Sigma$ whose associated open book decomposition induces a tight contact structure $\xi_{L}$.  Then $\Sigma$ maximizes Euler characteristic amongst all  surfaces in  $Y\times [0,1]$ which are smoothly embedded, have boundary $L$, and carry the relative homology class $[\Sigma]$. \end{conjecture}

To see that Conjecture \ref{conj1} implies Conjecture \ref{conj2}, it suffices to observe that the boundary of the fiber is naturally a transverse link in the contact structure induced by the  open book decomposition, and the self-linking number of this transverse link is given by negative the Euler characteristic of the fiber: \[ sl_{\xi_{L}}(L)=-\chi(\Sigma).\]  Then we can approximate $L$ by a Legendrian curve $\mathcal{L}$ with  \[ tb_\xi(\mathcal{L})+ rot_{[\Sigma]}(\mathcal{L}) =sl_{\xi_{L}}(L).\]  If Conjecture \ref{conj1} holds it follows that the Euler characteristic of any other slice surface for $L$ in the relative homology class of the fiber will be less than or equal to $\chi(\Sigma)$.

By the same reasoning we see that our main theorem establishes the second conjecture in the case that the contact structure has non-trivial invariant:

\begin{thm}\label{milnor}  Suppose $L\subset Y$ is a fibered link  whose associated contact structure has non-vanishing \os\ invariant.  Then the fiber maximizes Euler characteristic amongst  
 all  surfaces in its relative homology class that are smoothly embedded within $Y\times [0,1]$ and have boundary $L$.
\end{thm}
 
Finally, we show that subsurfaces of open books inducing contact structures with non-trivial contact invariant are slice-genus minimizing for their boundary:
 
 \begin{thm}\label{subsurface}  Suppose  a link $L\subset Y$ embeds as a separating curve on the page of an open book inducing a contact structure  with non-vanishing \os\ invariant.  Then the subsurface  of the page bounded by $L$ maximizes Euler characteristic amongst  
 all surfaces in its relative homology class that are smoothly embedded in $Y\times [0,1]$ with boundary $L$.
\end{thm}

Hence for this class of links the 4-dimensional complexity agrees with the Thurston norm.  One can think about such links as an analogue in general 3-manifolds of Rudolph's {\em strongly quasipositive links}.  This latter class, while initially defined using certain band presentations of surfaces, was shown to be exactly the class of links in the 3-sphere which embed as separating curves on the fiber surface of a torus link (or any fibered link inducing the standard tight contact structure).  See \cite[Proposition 4.2]{Hayden} for a related result.   We conclude with the following conjecture, which sits naturally with the others. 

\begin{conjecture} \label{conj3} A closed contact manifold is tight if and only if every subsurface of the page of each of its supporting open books is Thurston norm minimizing and, moreover, is minimal complexity for its boundary amongst all surfaces in $Y\times [0,1]$. 
\end{conjecture}
 
\noindent The surfaces considered by the conjecture should share the same relative homology class as the subsurface in question, just as in Conjecture \ref{conj2}.   Using Eliashberg's classification of overtwisted contact structures, together with the fact that one can find open books inducing the overtwisted contact structure on $S^3$ with Hopf invariant zero whose pages do not minimize slice genus, one can easily establish the ``only if" direction of the conjecture. \\

\noindent{\bf Acknowledgements:}  We thank Tom Mrowka, Danny Ruberman, and John Etnyre for interesting conversations.
 
\section{Proof of Theorems}
In this section we prove the theorems stated in the introduction.     The key tool is a numerical invariant $\tau_\xi(Y,K)\in \Z$, of a knot $K$ in a contact manifold $(Y,\xi)$ defined by the first author in \cite{tbbounds}. Roughly speaking, this invariant measures the Alexander filtration level in the knot Floer homology filtration when the contact invariant first appears. It generalizes the \os-Rasmussen concordance invariant $\tau(K)$ for knots in $S^3$ and satisfies a Bennequin-type inequality analogous to one proved by Plamenevskaya in that setting \cite{Olga2004}.   The strategy will be to combine the Bennequin-bound for $\tau_\xi(Y,K)$ with a relative adjunction inequality for the generalized $\tau$ invariants recently established by the authors in \cite{adjunction}.  Combined, and correctly interpreted, these two inequalities will quickly yield Theorem \ref{sB} for surfaces with connected boundary; the extension to surfaces whose boundaries are multi-component Legendrian links will be deduced from properties of a ``knotification" operation introduced by \ons\ \cite[Section 2.1]{Knots}.  As discussed in the introduction, Theorem \ref{milnor} follows immediately from Theorem \ref{sB}. We then prove Theorem \ref{subsurface} using a cobordism argument.   Along the way, we prove slice-Bennequin inequality for rationally null-homologous knots, Theorem \ref{thm:rational}, and establish a Heegaard Floer analogue of  Mrowka-Rollin's bound, Theorem \ref{MRbound}. \\

We begin in earnest by recalling the definition of $\tau_\xi(Y,K)$. We assume that knots are  null-homologous, unless otherwise specified.   We will not review the basics of knot Floer homology or the contact invariant; for this, we refer  to \cite{tbbounds} for a thorough introduction aimed at the present applications.   Let us recall, though, that the contact invariant $c(\xi)$ of a contact manifold $(Y,\xi)$ resides in the Floer homology of the manifold with its orientation reversed $c(\xi)\in \HFa(-Y)$ \cite[Definition 1.2]{Contact}, a group which can be identified with the Floer cohomology of $Y$ \cite[Proposition 2.5]{HolDiskTwo}.  As such, there is a bilinear hom pairing (we work with $\F=\Z/2\Z$ coefficients throughout)
\[ \HFa(-Y)\otimes \HFa(Y)\rightarrow \F\]
which we denote by $\langle - , - \rangle$.  We have the following invariant
\begin{defn}\cite[Definition 17 and 23]{tbbounds}  Let $K\subset Y$ be a knot and $\xi$ a contact structure on $Y$ with $c(\xi)\ne 0$.  
\[	\tau_\xi(Y,K):=\mathrm{min}\{m\in\Z|\ \exists \alpha\in\mathrm{Im} (H_*\Filt(K,m)\rightarrow \HFa(Y)) \ \mathrm{such \ that \ } \langle c(\xi),\alpha\rangle \ne 0\}. \]
\end{defn}

It should be noted that the Alexander grading on knot Floer homology, and hence the index for the associated filtration $\Filt(K)$, depends on the relative homology class $[S,K]\in H_2(Y,K)$ of a choice of (rational) Seifert surface.  We suppress this dependence throughout.

The motivation for the invariant above is the following bound it provides on the Thurston-Bennequin and rotation numbers \cite[Theorem 2]{tbbounds}:
\begin{equation}\label{eq:taubennequin} tb(\mathcal{K})+|rot_\SSurf(\mathcal{K})|\le 2\tau_\xi(Y,K)-1,\end{equation}
where $\mathcal{K}$ is any Legendrian representative of $K$ in $\xi$. We wish to combine this inequality with a relative adjunction inequality for generalized $\tau$ invariants, recently proved by the authors in \cite{adjunction}.  For this, define: 
\begin{defn} Let $K\subset Y$ be a knot and $\alpha\in \HFa(Y)$ a non-zero class.  
\[	\tau_\alpha(Y,K):=\mathrm{min}\{m\in\Z|\ \alpha\in\mathrm{Im} (H_*\Filt(K,m)\rightarrow \HFa(Y))\}. \]
\end{defn}

The main theorem of \cite[Theorem 1]{adjunction} indicates that  $\tau_\alpha(Y,K)$  constrains the genera of smooth surfaces bounded by $K$ in cobordisms whose map on Floer homology involves the class $\alpha$ in a non-trivial way.  To state it, suppose more generally that we are given two oriented knots $K_0 \subset Y_0$, $K_1 \subset Y_1$, and an oriented   $\SpinC$ cobordism $(W,\spinct)$  between their ambient 3-manifolds.  If, for classes $\alpha\in \HFa(Y_0), \beta\in \HFa(Y_1)$ the induced map on Floer homology satisfies $\widehat{F}_{W,\spinct}(\alpha)=\beta$, then
\begin{equation}\label{eq:adj} \langle c_1(\spinct),[ \Sigma_{S_0,S_1}]\rangle + [ \Sigma_{S_0,S_1}]^2 + 2(\tau_\beta(Y_1,K_1)-\tau_\alpha(Y_0,K_0)) \le 2g(\Sigma),\end{equation}
where $\Sigma$ is any smooth oriented surface in $W$ with oriented boundary $-K_0\sqcup K_1$, and $[ \Sigma_{S_0,S_1}]\in H_2(W)$ is the (absolute) homology class resulting from capping $\Sigma$ with Seifert surfaces $S_0,-S_1$ whose boundary are  $K_0, -K_1$ respectively.   

Finally, we have the following proposition relating the $\tau$ invariant associated to the contact class  to the invariant $\tau_\xi(Y,K)$. 
\begin{prop}\cite[Definition 23 and Proposition 28]{tbbounds}\label{Prop:reverseorient}
Let $c(\xi)\in\HFa(-Y)$ be the contact class. If $c(\xi)\neq 0$ then \[\tau_{c(\xi)}(-Y,K)=-\tau_{\xi}(Y,K).\]
\end{prop} 

\begin{proofof}{\bf Theorem \ref{sB}.} 

Suppose $(Y,\xi)$ is a contact manifold with non-vanishing \os \ contact invariant $c(\xi)\ne 0\in \HFa(-Y)$. We begin by assuming that the surface $\Sigma$ has connected boundary, a knot $K$ which is Legendrian with respect to $\xi$.  We wish to apply the relative adjunction inequality \eqref{eq:adj}  to the 4-manifold $Y\times [0,1]$ to bound the genus of $\Sigma$ in terms of $\tau_\xi(Y,K)$.  Coupled with the Bennequin bound \eqref{eq:taubennequin} provided by $\tau_\xi(Y,K)$, this would prove the theorem for surfaces with connected boundary were it not for a wrinkle in the definitions arising from the fact that $\tau_\xi(Y,K)$ is defined in a dual manner to $\tau_\alpha(Y,K)$.  

To iron this out, we observe that $\partial(Y\times [0,1])=-Y\times \{0\}\sqcup Y\times \{1\}$. Thus we may regard $Y\times [0,1]$ as a cobordism from $-Y\times\{1\}$ to $-Y\times \{0\}$  by ``turning the cobordism around". Regarding our manifold this way, we have that the induced map on Floer homology $\widehat{F}_{Y\times I}:\HFa(-Y)\to\HFa(-Y)$ sends the contact invariant $c(\xi)$ to itself. 

Now we treat $\Sigma\subset Y\times [0,1]$ as a cobordism from $\partial\Sigma\subset -Y\times \{1\}$ to the unknot $U\subset -Y\times\{0\}$. by tubing $\Sigma$ to the other boundary component. We then observe that the unknot has all $\tau$ invariants zero, when defined with respect to the relative homology class of a Seifert surface provided by a disk.  Finally, we note that if we define the Alexander grading for $K$ with respect to the the homology class of $\Sigma$, then capping with a Seifert surface in (negative) this relative class produces a null-homologous closed surface.  Thus the Chern class evaluation and self-intersection number in \eqref{eq:adj} vanish, and the relative adjunction inequality  reduces to \[ 0-2\tau_{c(\xi)}(-Y,K) \le 2g(\Sigma).\] Finally, since \[-\tau_{c(\xi)}(-Y,K)=\tau_{\xi}(Y,K)\] we obtain the desired inequality.

We now turn to the case that the surface $\Sigma$ has multiple boundary components; that is, $\partial \Sigma=L\subset Y$ is a link. To treat links, we appeal to \ons's ``knotification" construction from \cite[Section 2]{Knots}, which associates a knot $\kappa(L)\subset Y\#^{|L|-1} S^1\times S^2$ to an $|L|$-component link $L\subset Y$.  The knot $\kappa(L)$ is defined by attaching bands which join the components of $L$, and which run through the boundary of $4$-dimensional $1$-handles attached to $Y$ with feet centered at the points of $L$ where the bands are attached.  We can think of this  as  attaching  $1$-handles in the category of manifold pairs;  $(\mathbb{D}^1\times \mathbb{D}^3, \mathbb{D}^1\times \mathbb{D}^1)$ attached along the positive boundary of $(Y,L)\times [0,1]$.  From this perspective, it is clear that the knot on the new boundary, $\kappa(L)$, is well-defined up to diffeomorphism: any ambiguity in the construction can be dispatched with by performing handles slides (see Proposition 2.1 of \cite{Knots} for more details).

We can now define $\tau_\xi(L)$ for links $L\subset Y$ in terms of knotification. To do so, we recall that there is a unique tight contact structure $\xi_{std}$ on $\#^{|L|-1}S^1\times S^2$ obtained as the boundary of the Stein domain gotten from the $4$-ball  by attaching Stein $1$-handles.  Attaching Stein $1$-handles to $Y\times [0,1]$ gives rise to a contact structure on $Y\#^{|L|-1} S^1\times S^2$ which is contactomorphic to the contact connected sum $\xi\#\xi_{std}$ (see \cite{Colin,EtnyreHondaConnected} for details on the contact connected sum operation).   We make the following definition:

\begin{defn}  Let $L\subset (Y,\xi)$ be a link in a contact manifold.  Then 
\[ \tau_\xi(Y,L):= \tau_{\xi\#\xi_{std}}(Y\#^{|L|-1} S^1\times S^2, \kappa(L))\]
\end{defn}

There is a K\"unneth theorem for the Floer homology of connected sums \cite[Theorem 1.4]{HolDiskTwo} which, applied to the situation at hand,  states that there is an isomorphism
\[ \HFa(Y\#^{|L|-1}S^1\times S^2,\spinc\#\spinc_0)\cong \HFa(Y,\spinc)\otimes_\F \HFa(\#^{|L|-1}S^1\times S^2,\spinc_0).\]
 
\noindent Under this, the contact invariants satisfy a product formula \cite[Property 4, pg. 20]{tbbounds} \[c(\xi\#\xi_{std})=c(\xi)\otimes c(\xi_{std}),\]  and $c(\xi_{std})\in \HFa(-\#^{|L|-1}S^1\times S^2)$ can easily be shown to be non-trivial; for instance, $\xi_{std}$ is Stein fillable, which implies non-triviality  \cite[Theorem 1.5]{Contact}. Thus, under the assumption that $c(\xi)\in\HFa(-Y)$ is non-trivial, the Bennequin inequality applies: 
\begin{equation}\label{eq:tbboundkappa} tb(\kappa(L))+rot(\kappa(L))\le 2\tau_{\xi\#\xi_{std}}(Y\#^{|L|-1} S^1\times S^2, \kappa(L))-1 := 2\tau_\xi(Y,L)-1.\end{equation}

On the other hand,  a Legendrian representative for $L$ naturally induces a Legendrian representative for $\kappa(L)$ satisfying:
\begin{equation}\label{eq:tbadditivity} tb(\kappa(L))=tb(L)+|L|-1 \ \ \ \  \ \ \ rot(\kappa(L))=rot(L).\end{equation}
This can be seen using the same proof as \cite[Lemma 3.3]{EtnyreHondaConnected}.  

Now, by construction, any surface $\Sigma\subset Y\times [0,1]$ with boundary $L$ induces a surface $\Sigma'=\Sigma\cup \{1$-handles$\}$   in   $Y\times [0,1] \cup\{1$-handles$\}$ bounded by $\kappa(L)$, with 
\begin{equation}\label{eq:chisigma} -\chi(\Sigma')=-\chi(\Sigma)+|L|-1.\end{equation} As before, we view the 4-manifold $X=Y\times [0,1] \cup \{1$-handles$\}$ as a cobordism from $-(Y\#^{|L|-1}S^1\times S^2)$ to $-Y$.  The associated map on Floer homology sends $c(\xi\#\xi_{std})$ to $c(\xi)$, as can be seen by direct computation.  Indeed, it is straightforward to verify that $c(\xi_{std})\in \HFa(-\#^{|L|-1} S^1\times S^2)$ is given by the dual of  $\Theta_{top}\in\HFa(\#^{|L|-1} S^1\times S^2)$, the generator  with highest Maslov grading; see, e.g. \cite[Proof of Proposition 5.19]{adjunction}.  The product formula for the contact invariant of a connected sum therefore yields $c(\xi\#\xi_{std})=c(\xi)\otimes\Theta_{top}^*$. By definition,  the map on Floer homology associated to a 1-handle cobordism takes $\alpha\in \HFa(Y)$ to $\alpha\otimes\Theta_{top}\in \HFa(Y\#^{|L|-1} S^1\times S^2)$, hence the dual map (which is identified with the  map associated to the cobordism turned around) sends $c(\xi\#\xi_{std})$ to $c(\xi)$ as asserted.   It then follows from the relative adjunction inequality that  
\begin{equation} \label{eq:kappaadjunction} -2\tau_{c(\xi\#\xi_{std})}(-(Y\#^{|L|-1}S^1\times S^2),\kappa(L)) -1 \le 2g(\Sigma')-1 = -\chi(\Sigma')  \end{equation}

 Stringing everything together yields  the desired inequality: 
 \[ tb(L)+|L|-1 + rot(L) \underset{\eqref{eq:tbadditivity}}= tb(\kappa(L))+rot(\kappa(L)) \underset{\eqref{eq:tbboundkappa}}\le 2\tau_{\xi\#\xi_{std}}(Y\#^{|L|-1}S^1\times S^2, \kappa(L))-1\]\[\underset{\mathrm{Prop.}\ \ref{Prop:reverseorient}}=-2\tau_{c(\xi\#\xi_{std})}(-(Y\#^{|L|-1}S^1\times S^2),\kappa(L)) -1 \underset{\eqref{eq:kappaadjunction}}\le -\chi(\Sigma') \underset{\eqref{eq:chisigma}}= -\chi(\Sigma)+|L|-1.\]
\end{proofof}

More generally, if we consider rationally null-homologous knots in contact 3-manifolds, we can prove a rational slice-Bennequin inequality.  To state it, recall from \cite{Baker-Etnyre} that Baker and Etnyre define invariants $tb_\Q(\mathcal K)$ and $rot_\Q(\mathcal K)$ associated to a rationally null-homologous Legendrian $\mathcal{K}$ in a contact 3-manifold $(Y,\xi)$. Moreover, they prove a version of the Eliashberg-Bennequin bound in this setting. Specifically, if $(Y,\xi)$ is tight and $K$ is any rationally null-homologous knot, \[tb_\Q(\mathcal K)+|rot_\Q(\mathcal{K})|\le -\frac{1}{q}\chi(F)\] where $\mathcal K$ is any Legendrian representative of $K$, $F$ is a rational Seifert surface for $K$, and $q$ is the order of $K$ in $H_1(Y;\Z)$. 
 
Before stating our rational slice-Bennequin bound, we make a definition. Let $K\subset Y$ be a knot whose homology class $[K]\in H_1(Y)$ has order $q$. A (smooth) \emph{rational slice surface}   for  $K$ consists of a compact oriented surface-with-boundary $\Sigma$,  along with a smooth map $\Sigma\to Y\times [0,1]$ that is an  embedding on the interior of $\Sigma$ and for which the restriction to $\partial \Sigma$ is a $q$-fold covering of $K\times\{1\}$. We let $\Sigma$ denote the singular surface in $Y\times [0,1]$ arising as the image of the defining map. One could, of course, consider the corresponding notion in the locally flat category.

\begin{theorem}\label{thm:rational}
Let $(Y,\xi)$ be a contact 3-manifold with non-trivial Ozsv\'ath-Szab\'o contact class. If $K\subset Y$ is a knot  of order $q$ and $\mathcal K$ is any Legendrian representative, then \[tb_{\Q}(\mathcal K)+rot_{\Q}(\mathcal K) \leq -\frac{1}{q}\chi(\Sigma)\] where $\Sigma\hookrightarrow Y\times[0,1]$ is any smooth rational slice surface for $K$.
\end{theorem}
\begin{proof} By Li and Wu \cite[Theorem 1.1]{Li-Wu}, if $\mathcal K$ is a Legendrian representative of $K$ then \[tb_\Q(\mathcal K)+rot_\Q(\mathcal{K})\le 2\tau_\xi(Y,K)-1.\]
The definition of $\tau_\xi(Y,K)$ here is the same as for a null-homologous knot, with the caveat that because $K$ is {rationally} null-homologous the Alexander grading may take {rational} values. Thus $\tau_\xi(Y,K)$ may be a rational number. See \cite{Li-Wu} or \cite{adjunction} for more details.   In \cite{rationalgenusbounds}, the authors prove an analogue of \eqref{eq:adj} which constrains the genera of cobordisms between cables of rationally null-homologous knots. In the case of $Y\times [0,1]$, this inequality constrains the Euler characteristic of rational slice surfaces by any of the invariants $\tau_\alpha(Y,K)$.  In particular, we have 
\[-2\tau_{c(\xi)}(-Y,K)-1\le -\frac{1}{q}\chi(\Sigma).\]
Combining this with $\tau_\xi(Y,K)=-\tau_{c(\xi)}(-Y,K)$ (Proposition \ref{Prop:reverseorient}) and Li and Wu's result yields the stated bound.  
\end{proof} 

\begin{remark} In the case that the ambient 3-manifold has $b_1(Y)>0$,  the Alexander grading $($and hence $\tau_\xi(Y,K))$ and rotation number will depend on the relative homology class of a choice of rational Seifert surface. However, the bound we obtain for $\chi(\Sigma)$ is independent of these choices; see \cite[Proof of Theorem 1]{adjunction} for more details.
\end{remark}

The techniques at hand  yield a Heegaard Floer analogue of  the Mrowka-Rollin bound. 
 
\begin{theorem}\label{MRbound} Suppose $(W,\spinct)$ is a $\SpinC$ cobordism satisfying $\widehat{F}_{W,\spinct}(c(\xi))\ne 0$, where $c(\xi)\in \HFa(-Y)$ is the contact invariant of a contact structure $\xi$ on $Y$.  Then for a null-homologous Legendrian knot $\mathcal{K}\subset Y$, and any smoothly embedded  surface $\Sigma\subset W$ with boundary $\mathcal{K}$, we have
\[  tb(\mathcal{K},[\Sigma])+|r(\mathcal{K},[\Sigma],\spinct,h)|\le -\chi(\Sigma),\]
where $tb(\mathcal{K},[\Sigma])$ and $r(\mathcal{K},[\Sigma],\spinct,h)$ denote Thurston-Bennequin and rotation numbers defined with respect to the relative homology class of $\Sigma$ and an identification $h$ between $\spinct|_Y$ and $\spinc_\xi$  (see \cite{MrowkaRollin}).
\end{theorem}
\noindent It is worth pointing out that in the context of null-homologous knots, the rotation number above is independent of $h$.

\begin{proof}  The assumption that $\widehat{F}_{W,\spinct}(c(\xi))\ne 0$ allows us to employ the relative adjunction inequality as before, with $-Y$ viewed as the incoming boundary of $W$:
\[   \langle c_1(\spinct),[ \Sigma_{S}]\rangle + [ \Sigma_{S}]^2 + 2(0-\tau_{c(\xi)}(-Y,K)) \le 2g(\Sigma).\]
Here, $\Sigma_{S}$ is the  homology class in $H_2(W)$ of $\Sigma\cup -S$, where $S$ is a Seifert surface for $K$.  The zero term is added to emphasize the direction of the cobordism, and can be interpreted as the $\tau$ invariant of an unknot in the outgoing boundary of $W$, tubed to $\Sigma$.  Using the equality $-\tau_{c(\xi)}(-Y,K)=\tau_{\xi}(Y,K)$ as before, together with the Bennequin bound  \eqref{eq:taubennequin} for $\tau_{\xi}$, we have: 
\[  \langle c_1(\spinct),[ \Sigma_{S}]\rangle + [ \Sigma_{S}]^2 + tb(\mathcal{K})+ rot_S(\mathcal{K}) \le \langle c_1(\spinct),[ \Sigma_{S}]\rangle + [ \Sigma_{S}]^2 + 2\tau_{\xi}(Y,K))-1 \le 2g(\Sigma)-1,\]
where $tb(\mathcal{K})$ and $rot_S(\mathcal{K})$ denote the Thurston-Bennequin and rotation numbers, as typically defined, which  agree in this case with Mrowka and Rollin's definition of $tb(\mathcal{K},[S])$ and $r(\mathcal{K},[S],\spinct,h)$.   By their definition, these latter invariants depend on the relative homology class in the following manner:
\[  tb(\mathcal{K},[\Sigma])- tb(\mathcal{K},[S]) = [\Sigma-S]^2, \ \ \ \ r(\mathcal{K},[\Sigma],\spinct,h)- r(\mathcal{K},[S],\spinct,h) = \langle c_1(\spinct),[ \Sigma-S]\rangle \]
The result follows.
\end{proof}

We should remark that the result above is both more and less general than Mrowka and Rollin's bound.  Theirs applies to arbitrary Legendrians, while our proof requires them  to be null-homologous (though the technique readily extends to rationally null-homologous knots, as in Theorem \ref{thm:rational}).  On the other hand, our result applies to 4-manifolds with multiple boundary components, and is stated in terms of non-vanishing of the cobordism-induced map on the ``hat" contact invariant.  Assuming a natural isomorphism between monopole Floer homology and Heegaard Floer homology with respect to cobordism maps, the Mrowka-Rollin theorem applies to cobordisms from $-Y$ to $-S^3$ mapping the ``plus" contact invariant $c^+(\xi)$ non-trivially, a stronger Floer theoretic assumption than ours, placed on a more limited collection of 4-manifolds.\\

We conclude with the proof of Theorem \ref{subsurface}.

\begin{proofof}{\bf Theorem \ref{subsurface}.}

Let $F$ be the page of the open book, and $S\subset F$ be the interior of the subsurface cut off from $\partial F$ by $L$. Then $\cob$ is a cobordism from $L$ to $\partial F$, smoothly embedded in $Y$.  Pick a Morse function  $f:\cob\rightarrow[0,1]$ which takes the value $0$ on $L$ and $1$ on $\partial F$, and use this function to produce a proper embedding of $\cob$ into $Y\times [0,1]$.  We first treat the simpler case that  $L$ is connected.  We can assume without loss of generality that $\partial F$ is also connected since, if it were not, we could plumb positive Hopf bands to the page until its boundary is connected without changing the fact that $S$ is a subsurface or that the contact invariant is non-trivial.

Assuming then that $\partial F$ is a knot,  Theorem $5$ of \cite{tbbounds} shows that $\tau_\xi(\partial F)=g(F)$, where  the contact structure $\xi$ used in the definition is that induced by the open book decomposition of $Y$ with page $F$. Here we are using the assumption that $c(\xi)\ne 0$ in the definition of $\tau_\xi$.  

Viewing $Y\times[0,1]$ as a cobordism from $-(Y\times\{1\})$ to $-(Y\times\{0\})$, and applying the relative adjunction inequality \eqref{eq:adj} yields:
\[ 2(\tau_{c(\xi)}(-Y,L)-\tau_{c(\xi)}(-Y,\partial F)) \le 2g(\cob),\]
which, using Proposition \ref{Prop:reverseorient},  implies
\[2(\tau_\xi(\partial F)-\tau_\xi(L)) \le 2g(\cob).\]
\noindent Rearranging, and using the fact that $\tau_\xi(\partial F)=g(F)$, we obtain:
\begin{equation} \label{1} 2g(S)= 2(g(F)-g(\cob)) \le 2\tau_\xi(L).  \end{equation}

\noindent But we could also apply the  relative adjunction inequality to any surface $\Sigma\subset Y\times[0,1]$ with  $\partial\Sigma=L\subset  Y\times\{1\}$  to produce the following inequality:
\begin{equation} \label{2} 2\tau_\xi(L) \le 2g(\Sigma). \end{equation}

\noindent Combining   Inequalities \eqref{1} and \eqref{2}, it follows that $S$ minimizes genus amongst all  surfaces in $Y\times [0,1]$ with boundary $L$.

We now turn to the case that $L=\partial S$ has $n>0$ components. In this case, we again plumb positive Hopf bands to $F$, but this time to ensure that the resulting surface  has $n$ boundary components.  As before, we can do this without changing the contact structure associated to the resulting open book, or the fact that $S$ embeds as a subsurface of the new page, which we will henceforth denote by $F$.  As before, we push  $\cob$ into $Y\times[0,1]$ to view it as a cobordism between $L\subset Y\times\{0\}$ and $\partial F\subset Y\times\{1\}$.  Now find an embedding of $2n-2$ arcs
\[\phi: \sqcup_{i=1}^{2n-2} [0,1]  \hookrightarrow \cob\subset Y\times [0,1] \]
so that post-composition with the projection $\pi :Y\times [0,1] \rightarrow [0,1]$ restricts to the identity on each arc, and so that each component but two of $L$ and $\partial F$ contain exactly two endpoints of (necessarily distinct) arcs. Hence there are  two components of $L$ and $\partial F$ which each contain exactly one endpoint.  See Figure \ref{fig:surgery} for a schematic.  

\begin{figure}
      \def\svgwidth{0.5\textwidth}
\begingroup%
  \makeatletter%
  \providecommand\color[2][]{%
    \errmessage{(Inkscape) Color is used for the text in Inkscape, but the package 'color.sty' is not loaded}%
    \renewcommand\color[2][]{}%
  }%
  \providecommand\transparent[1]{%
    \errmessage{(Inkscape) Transparency is used (non-zero) for the text in Inkscape, but the package 'transparent.sty' is not loaded}%
    \renewcommand\transparent[1]{}%
  }%
  \providecommand\rotatebox[2]{#2}%
  \newcommand*\fsize{\dimexpr\f@size pt\relax}%
  \newcommand*\lineheight[1]{\fontsize{\fsize}{#1\fsize}\selectfont}%
  \ifx\svgwidth\undefined%
    \setlength{\unitlength}{899.80933963bp}%
    \ifx\svgscale\undefined%
      \relax%
    \else%
      \setlength{\unitlength}{\unitlength * \real{\svgscale}}%
    \fi%
  \else%
    \setlength{\unitlength}{\svgwidth}%
  \fi%
  \global\let\svgwidth\undefined%
  \global\let\svgscale\undefined%
  \makeatother%
  \begin{picture}(1,0.50473474)%
    \lineheight{1}%
    \setlength\tabcolsep{0pt}%
    \put(0,0){\includegraphics[width=\unitlength,page=1]{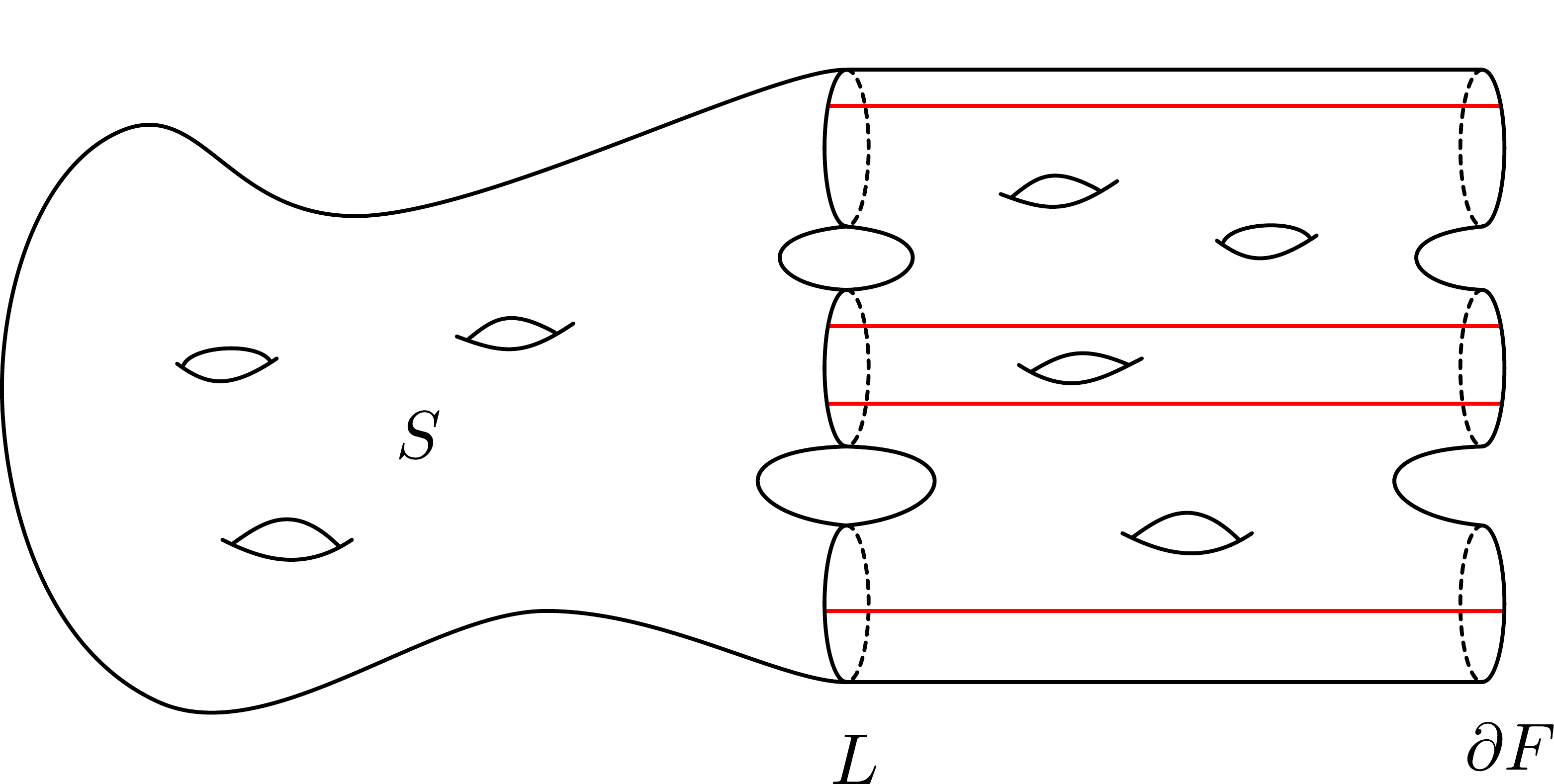}}%
    \put(0.67980517,0.4705502){\color[rgb]{0,0,0}\makebox(0,0)[lt]{\lineheight{1.25}\smash{\begin{tabular}[t]{l}$F\setminus S$\end{tabular}}}}%
  \end{picture}%
\endgroup%

 \caption{ \label{fig:surgery}  }
\end{figure}

Pairing the arcs in such a way that no pair has endpoints on the same component of $L$ or $\partial F$, we perform $n-1$ surgeries along  neighborhoods of the pairs;  that is, we remove neighborhoods of pairs of arcs, diffeomorphic to $B^3\times [0,1]\sqcup B^3\times [0,1] $, and replace them with $S^2\times [0,1]\times [0,1]$.  The resulting $4$-manifold is diffeomorphic to $(Y\#^{n-1}S^1\times S^2)\times [0,1]$.  In addition, we surger $\cob$ along the pairs of arcs, replacing the neighborhoods of the arcs with $\{p\}\times [0,1]\times [0,1]$ where $p\in S^2$.  This has no effect on the Euler characteristic of $\cob$, but produces a new surface, which we denote $\cobprime$, which is smooth and properly embedded in  $(Y\#^{n-1}S^1\times S^2)\times [0,1]$, and whose boundary components are the knotifications $\kappa(L)$ and $\kappa(\partial F)$ of $L$ and $\partial F$, respectively.   

We now observe that plumbing copies of the annulus open book with identity monodromy to $F$ in such a way as to pair up its boundary components yields an open book $\widehat{F}$ for $Y\#^{n-1}S^1\times S^2$ with  binding isotopic to $\kappa(\partial F)$ and with Euler characteristic $\chi(\widehat{F})=\chi(F)-n+1$.  Moreover, the contact structure on $Y\#^{n-1}S^1\times S^2$ induced by $\widehat{F}$ is isotopic to the contact connected sum of  $\xi$ with the unique tight contact structure $\xi_{std}$ on $\#^{n-1}S^1\times S^2$.  The resulting contact invariant is given by $c(\xi)\otimes c(\xi_{std})\in \HFa(-Y)\otimes\HFa(-\#^{n-1}S^1\times S^2)\cong \HFa(-(Y\#^{n-1}S^1\times S^2)),$ and is therefore non-zero.    At this point, we can  appeal as above to the genus minimizing property of the page of an open book with non-trivial contact invariant, to show that
\[ \tau_{\xi\#\xi_{std}}(\partial \widehat{F})=g(\widehat{F}),\]
which we express in terms of Euler characteristic as
\begin{equation} \label{3}  2\tau_{\xi\#\xi_{std}}(\partial \widehat{F}) = -\chi(\widehat{F}) + 1 = -\chi(F)+n-1+1=-\chi(F)+n  \end{equation}

Finally, we apply the relative adjunction inequality to the cobordism $\cobprime$ between $\kappa(L)$ and $\kappa(\partial F)$:
\begin{equation} \label{4}  2\tau_{\xi\#\xi_{std}}(\kappa(\partial F))-2\tau_{\xi\#\xi_{std}}(\kappa(L)) \le 2g(\cobprime) = -\chi(\cobprime)=-\chi(\cob)\end{equation}

Combining \eqref{3} and \eqref{4}, we obtain
\[   -\chi(S)+n= -\chi(F)+n + \chi(F\setminus S)  \le 2\tau_{\xi\#\xi_{std}}(\kappa(L))  \]

\noindent But $-\chi(S)+n$ is equal to twice the genus of the surface $\widehat{S}$ obtained by attaching bands to $L$ along $\partial S$ inside the 1-handles used in the knotification.  Applying the relative adjunction inequality to any surface $\Sigma$ with boundary $\kappa(L)$ inside the $4$-manifold $Y\times[0,1] \cup \{1$-handles$\}$, we have:
\[ -\chi(S)+n = 2g(\widehat{S})\le 2\tau_{\xi\#\xi_{std}}(\kappa(L)) \le 2g(\Sigma) \]

Thus, any surface $\Sigma$ bounded by $\kappa(L)$ in $Y\times[0,1] \cup \{1$-handles$\}$ satisfies $-\chi(S)\le 2g(\Sigma)-n$.  It follows at once that if $G\subset Y\times [0,1]$ were a surface bounded by $L$ with $-\chi(G)<-\chi(S)$, then attaching bands to $G$ would result in a surface violating $-\chi(S)\le 2g(\Sigma)-n$.  Thus $S$ maximizes Euler characteristic amongst all surfaces in $Y\times[0,1]$ sharing its boundary.
\end{proofof}
\begin{remark} One can provide an alternative proof of Theorem \ref{subsurface} by showing that the boundary of subsurfaces of the page of an open book can be perturbed to transverse links for which the Bennequin bound is sharp i.e. for which $sl_S(\partial S)= -\chi(S)$.  Combined  with the slice-Bennequin bound established by Theorem \ref{sB} for contact structures with non-trivial invariant, this establishes Theorem \ref{subsurface}.  We opt for the cobordism argument given above, as  the surgery technique used to deal with multiple boundary components should prove useful in a variety of contexts involving the $\tau$ invariants for links defined via knotification. 
\end{remark}
\bibliographystyle{plain}
\bibliography{4DTight}

\end{document}